\documentclass[11pt,reqno]{amsart}

\usepackage[utf8]{inputenc}
\usepackage[T1]{fontenc}
\usepackage{latexsym}
\usepackage{enumitem}
\usepackage{amsxtra}
\usepackage{amsmath}
\usepackage{amsfonts}
\usepackage{amssymb}
\usepackage{blkarray}
\usepackage{multirow}
\usepackage{mathrsfs}
\usepackage{mathdots}
\usepackage{pdfpages}
\usepackage{cite}
\usepackage{comment}
\usepackage{hyperref}
\usepackage[capitalise]{cleveref}

\usepackage[centering,margin=1in]{geometry}

\newtheorem{theorem}{Theorem}[section]
\newtheorem{lemma}[theorem]{Lemma}

\newtheorem{proposition}[theorem]{Proposition}

\newtheorem{claim}{Claim}[theorem]

\numberwithin{equation}{section}

\theoremstyle{definition}
\newtheorem{definition}[theorem]{Definition}

\theoremstyle{remark}

\usepackage{pgf,tikz}
\tikzstyle{vertex}=[circle, draw, inner sep=0pt, minimum size=6pt]

\DeclareMathOperator{\cl}{cl}

\setenumerate{label=\rm(\roman*),midpenalty=2}

\begin{document}
\title[The Minors of Matroids with an Adjoint]{The Minors of Matroids with an Adjoint}

\author[K.~Grace]{Kevin Grace}
\address{Department of Mathematics and Statistics\\
University of South Alabama\\
Mobile, Alabama}
\email{kevingrace@southalabama.edu}

\date{\today}

\begin{abstract}
  If $M$ is a matroid, then a simple matroid $M'$ with the same rank as $M$ is an \emph{adjoint} of $M$ if there is an inclusion-reversing embedding $\phi$ of the lattice of flats of $M$ into the lattice of flats of $M'$ such that $\phi$ maps the hyperplanes of $M$ bijectively onto the points of $M'$. In this note, we provide a proof that the class of matroids with an adjoint is minor-closed.
\end{abstract}

\maketitle

\section{Introduction}

The notion of an adjoint of a matroid seems to have first been introduced by Crapo \cite{c71}. However, as shown by Cheung \cite{Cheung}, not every matroid has an adjoint. Interest in adjoints of matroids has continued into recent years \cite{bc88,akw90,hk96,fjk23,ftw24}. It has been claimed that the class of matroids that have an adjoint is closed under the taking of minors \cite{hk96}; however, there does not seem to be an explicit proof of this in the literature. (However, see \cite{k95} and the note in the Acknowledgments below.)  In this note, we provide a proof of that fact. Additionally, combining \cref{prop:Closed-cont} with \cref{lem:closed-del} gives an explicit construction for finding an adjoint of a minor of $M$ when one knows an adjoint of $M$.

It should be noted that closure under taking of minors is in contrast to closure under duality. Alfter, Kern, and Wanka \cite{akw90} showed that the dual of a matroid with an adjoint can fail to have an adjoint itself.

The proof of our main result will be in \cref{sec:proof}. Before that, we recall some necessary matroid results in \cref{sec:Preliminaries}.

\section{Preliminaries}
\label{sec:Preliminaries}

We assume familiarity with matroid theory. Matroid terminology and notation not defined here will follow Oxley~\cite{o11}.

The precise definition of an adjoint of a matroid follows.
\begin{definition}
\label{def:adjoint_map}
    Let $M$ be a matroid, and let $M'$ be a simple matroid with the same rank as $M$. Let $\mathcal{F}(M)$ and $\mathcal{F}(M')$ be the collections of flats of $M$ and $M'$, respectively. Then $M'$ is an \emph{adjoint} of $M$ if there is a map $\phi:\mathcal{F}(M)\rightarrow\mathcal{F}(M')$ such that
    \begin{enumerate}
    \item $\phi$ is injective, 
    \item if $F_1,F_2\in\mathcal{F}(M)$ and $F_1\subseteq F_2$, then $\phi(F_2)\subseteq \phi(F_1)$, and
    \item the restriction of $\phi$ that maps hyperplanes of $M$ to points of $M'$ is bijective.
    \end{enumerate}
    We call $\phi$ an \emph{adjoint map} from $M$ to $M'$.
\end{definition}

The first statement in the next result was observed in \cite[Lemma 2.3]{bc88}, and the second statement was shown in \cite[Lemma 1]{hk96}. (Also, see \cite[Proposition 2.3]{ftw24}.)

\begin{theorem}
\label{thm:cited_theorem}
Let $M$ be a matroid with adjoint $M'$. The adjoint map $\phi:\mathcal{F}(M)\rightarrow\mathcal{F}(M')$ has the following properties:
    \begin{enumerate}
    \item \label{r-k} For each flat $F\in\mathcal{F}(M)$, $r_{M'}(\phi(F))=r(M)-r_M(F)$.
    \item \label{ind set} If $H_1,H_2,\dots,H_m$ are hyperplanes of $M$ such that \[\bigcap_{i=1}^{k+1} H_i\subsetneq\bigcap_{i=1}^{k} H_i\] for $1\leq k\leq m-1$, then $\{\phi(H_1),\phi(H_2),\dots,\phi(H_m)\}$ is an independent set in $M'$.
    \end{enumerate}
\end{theorem}

From \cite{o11}, we find several well-known results that will be used in the proof of our main result. The first of these results is from \cite[Proposition 3.3.7]{o11}.

\begin{proposition}
\label{prop:flats_after_contract_set}
    Let $M$ be a matroid. If $C\subseteq E(M)$, then $F$ is a flat of $M/C$ if and only if $F\cup C$ is a flat of $M$.
\end{proposition}

Next, we have \cite[Lemma 3.3.2]{o11}.
 
\begin{lemma}
\label{lem:ind-coind}
Every minor of a matroid $M$ can be written in the form $M/C\backslash D$, where $C$ and $D$ are independent and coindependent, respectively, in $M$.
\end{lemma}

Finally, the next result comes from \cite[Proposition 1.7.8]{o11} and additional details implicit in its proof.

\begin{proposition}
\label{pro:intersection-of-hyperplanes} 
Let $X$ be a flat of rank $k<r$ in a matroid $M$ of rank $r$. Then there is a collection of hyperplanes $\{H_1,H_2,\dots,H_{r-k}\}$ whose intersection is $X$. Moreover, \[\bigcap_{i=1}^{j+1} H_i\subsetneq\bigcap_{i=1}^{j} H_i\] for $1\leq j\leq r-k-1$.
\end{proposition}

\section{The Proof}
\label{sec:proof}

The next proposition follows fairly easily from the fact that the lattice of flats of a matroid $M/C$ is isomorphic to the interval of the lattice of flats of $M$ between $C$ and $E(M)$.

\begin{proposition}
\label{prop:Closed-cont}
Let $M$ be a matroid with an adjoint map $\phi$ to $M'$, and let $C\subseteq E(M)$. Then there is an adjoint map $\phi_{/C}$ from $M/C$ to $M'|\phi(\textnormal{cl}_M(C))$ given by $\phi_{/C}(F)=\phi(F\cup C)$, for each flat $F\in\mathcal{F}(M/C)$.
\end{proposition}

\begin{proof}
It follows from \cref{prop:flats_after_contract_set} that the map $\phi_{/C}$ is well-defined. We need to verify that $\phi_{/C}$ satisfies the three properties of \cref{def:adjoint_map}. To verify (i), note that every flat of $M/C$ is disjoint from $C$. Therefore, if $F_1$ and $F_2$ are flats of $M/C$ such that $F_1\cup C=F_2\cup C$, we must have $F_1=F_2$. Therefore, since $\phi$ is injective, so is $\phi_{/C}$.

Since $\phi$ satisfies (ii) of \cref{def:adjoint_map}, it is clear that $\phi_{/C}$ does as well. 

Now, to verify (iii), note that a set $H$ is a hyperplane of $M/C$ if and only if $H\cup C$ is a hyperplane of $M$. (This is \cite[3.1.19]{o11}.) Because of property (ii) of \cref{def:adjoint_map}, the hyperplanes of $M$ that contain $C$ are mapped by $\phi$ to the points of $M'$ that are contained in $\phi(\textnormal{cl}_M(C))$. Moreover, by (iii), this is done bijectively. Therefore, the restriction of $\phi$ to the hyperplanes of $M$ containing $C$ is a bijective map to the set of points of $M'$ contained in $\phi(\textnormal{cl}_M(C))$. The points of $M'$ contained in $\phi(\textnormal{cl}_M(C))$ are, of course, the points of $M'|\phi(\textnormal{cl}_M(C))$. This verifies (iii) and completes the proof.
\end{proof}

Although it is straightforward to show that the class of matroids with an adjoint is closed under contraction, it requires a bit more work to show that it is closed under deletion. If $M$ is a matroid and $D\subseteq E(M)$, then the collection of flats of $M\backslash D$ consists of the sets $F\subseteq E(M)-D$ such that there is a flat $F'$ of $M$ with $F=F'-D$ (see \cite[Proposition 3.3.7(ii)]{o11}). Now, suppose that $M$ has two flats $F'\subsetneq F''$ such that $F=F'-D=F''-D$. Then $r_{M\backslash D}(F)=r_M(F)\leq r_M(F')<r_M(F'')$. Thus, since $r(F''-D)<r(F'')$, the effect of deletion of $D$ on $F''$ is that it causes $F''$ to ``disappear'' from the lattice of flats. The next definition says that $\mathcal{H}(M,D)$ consists of the hyperplanes of $M$ that ``disappear'' from the lattice of flats when $D$ is deleted. \begin{definition}
    Suppose $M$ is a matroid and $D\subseteq E(M)$. If $\mathcal{H}(M)$ is the set of hyperplanes of $M$, let $\mathcal{H}(M,D)=\{H\in\mathcal{H}(M):r_M(H-D)<r_M(H)\}$.
\end{definition} Rewording this definition yet one more way, $\mathcal{H}(M,D)$ is the set of hyperplanes $H$ of $M$ for which $H\cap D$ is a codependent set in the restriction $M|H$.

If $M$ has an adjoint $M'$, the points of $M'$ correspond to the hyperplanes of $M$. Therefore, when deleting $D$ from $M$, it is intuitive to think that we can get an adjoint of $M\backslash D$ by deleting $\mathcal{H}(M,D)$ from $M'$. However, it is not immediately clear that each flat of $M\backslash D$ has a distinct flat of $M'\backslash\mathcal{H}(M,D)$ to which to be mapped. This is what the next lemma establishes.

\begin{lemma}
\label{lem:closed-del}
    Let $M$ be a matroid with an adjoint map $\phi$ to $M'$. If $D$ is a coindependent set in $M$, then there is an adjoint map $\phi_{\backslash D}$ from $M\backslash D$ to $N=M'\backslash\phi(\mathcal{H}(M,D))$ given by $\phi_{\backslash D}(F)=\phi(\cl_M(F))-\phi(\mathcal{H}(M,D))$ for flats $F$ of $M\backslash D$. 
\end{lemma}

\begin{proof}
Since the closure of a set is a flat, it follows that the map $\phi_{\backslash D}$ is well-defined. We need to verify that $\phi_{\backslash D}$ satisfies the three properties of \cref{def:adjoint_map}. Since $\phi$ satisfies (ii) of \cref{def:adjoint_map}, it is clear that $\phi_{\backslash D}$ does as well.

Since $D$ is coindependent, the ranks of $M$ and $M\backslash D$ are equal. Therefore, the ranks of their hyperplanes are equal. Thus, for each hyperplane $H$ of $M\backslash D$, we have that $\cl_M(H)$ is a hyperplane of $M$. Note that $H=\cl_M(H)-D$. Since $\cl_M(H)$ and $H$ have the same rank, $\cl_M(H)\notin\mathcal{H}(M,D)$. Therefore, $\phi_{\backslash D}(H)=\phi(\cl_M(H))\in E(N)$. Thus, since $\phi$ satisfies (iii) of \cref{def:adjoint_map}, it follows that $\phi_{\backslash D}$ does as well.

Now, we must show that $\phi_{\backslash D}$ is injective. Let $\alpha:\mathcal{F}(M\backslash D)\rightarrow\mathcal{F}(M)$ be defined by $\alpha(F)=\cl_M(F)$, and note that $\phi_{\backslash D}(F)=(\phi\circ\alpha)(F)-\phi(\mathcal{H}(M,D))$. Also, note that $r_{M\backslash D}(F)=r_M(\alpha(F))$ for every flat $F$ of $M\backslash D$. Let $r=r(M)=r(M\backslash D)$.

\begin{claim}
\label{rank-claim}
    If $F$ is a flat of $M\backslash D$ of rank $k$, then $\phi_{\backslash D}(F)$ is a flat of $N$ of rank $r-k$.
\end{claim}

\begin{proof}
By \cref{pro:intersection-of-hyperplanes}, there are $r-k$ hyperplanes $H_1,H_2,\dots,H_{r-k}$ of $M\backslash D$, whose intersection is equal to $F$, such that \[\bigcap_{i=1}^{j+1} H_i\subsetneq\bigcap_{i=1}^{j} H_i\] for $1\leq j\leq r-k-1$. (If $F$ is the ground set of $M\backslash D$, we take $F$ to be the intersection of an empty collection of hyperplanes.) We also have that $\alpha(F)$ is the intersection of $\alpha(H_1),\alpha(H_2),\dots,\alpha(H_{r-k})$.

Since $r_{M}(\alpha(H_i))=r_{M\backslash D}(H_i)$ for each $i$, each $\alpha(H_i)$ is a hyperplane of $M$. Moreover, $\alpha(H_i)=\cl_M(H_i)$, and $H_i=\alpha(H_i)-D$. Since the ranks of $H_i$ and $\alpha(H_i)$ are equal, $\alpha(H_i)\notin\mathcal{H}(M,D)$. Therefore, for each $i$, we have that $(\phi\circ\alpha)(H_i)$ is a point in both $M'$ and $N=M'\backslash\phi(\mathcal{H}(M,D))$. Since $\alpha(F)\subseteq\alpha(H_i)$ for each $i$, and since $\phi$ is inclusion-reversing, we have $(\phi\circ\alpha)(H_i)\subseteq(\phi\circ\alpha)(F)$ for each $i$. Therefore, $\{(\phi\circ\alpha)(H_1),\dots,(\phi\circ\alpha)(H_{r-k})\}\subseteq(\phi\circ\alpha)(F)-\phi(\mathcal{H}(M,D))$. 

By \cref{thm:cited_theorem}\ref{r-k}, we have $r_{M'}((\phi\circ\alpha)(F))=r-k$. It follows from \cref{thm:cited_theorem}\ref{ind set} that $r_{M'}(\{(\phi\circ\alpha)(H_1),(\phi\circ\alpha)(H_2),\dots,(\phi\circ\alpha)(H_{r-k})\})=r-k$.
Therefore, \[r_N(\phi_{\backslash D}(F))=r_N((\phi\circ\alpha)(F)-\phi(\mathcal{H}(M,D)))\geq r_{M'}(\{(\phi\circ\alpha)(H_1),\dots,(\phi\circ\alpha)(H_{r-k})\})=r-k.\]

Also, since $\phi_{\backslash D}(F)=(\phi\circ\alpha)(F)-D$, we have \[r_N(\phi_{\backslash D}(F))\leq r_{M'}((\phi\circ\alpha)(F))=r-k.\] Therefore, $r_N(\phi_{\backslash D}(F))=r-k$, as claimed.
\end{proof}

Note that $\phi\circ\alpha$ is injective since $\alpha$ and $\phi$ are both injective. To show that $\phi_{\backslash D}$ is injective, it remains to show that distinct members of the set $(\phi\circ\alpha)(\mathcal{F}(M\backslash D))\subseteq\mathcal{F}(M')$ cannot result in the same flat of $N$ when the set $\phi(\mathcal{H}(M,D))$ is deleted. Suppose for a contradiction that this is not true. That is, there are distinct flats $F$ and $F'$ of $M'$, with $\{F,F'\}\subseteq(\phi\circ\alpha)(\mathcal{F}(M\backslash D))$ such that $F-\phi(\mathcal{H}(M,D))=F'-\phi(\mathcal{H}(M,D))$. Let $F_1=F\cap F’$ and $F_2=F$. Then $r(F_1)<r(F_2)$ and $F_1-\phi(\mathcal{H}(M,D))= F_2-\phi(\mathcal{H}(M,D))$. Let $X_1$ and $X_2$ be the flats of $M\backslash D$ such that $(\phi\circ\alpha)(X_1)=F_1$ and $(\phi\circ\alpha)(X_2)=F_2$. Then $r_{M\backslash D}(X_1)>r_{M\backslash D}(X_2)$. By \cref{rank-claim}, since $\phi_{\backslash D}(X_1)=F_1-\phi(\mathcal{H}(M,D))$ and $\phi_{\backslash D}(X_2)=F_2-\phi(\mathcal{H}(M,D))$, we have $r_N(F_1-\phi(\mathcal{H}(M,D)))<r_N(F_2-\phi(\mathcal{H}(M,D)))$. This contradicts the assumption that $F_1-\phi(\mathcal{H}(M,D))=F_2-\phi(\mathcal{H}(M,D))$. Therefore, we have proved that $\phi_{\backslash D}$ is injective, completing the proof of the lemma.
\end{proof}

We can now prove our main result.

\begin{theorem}
\label{thm:minor-closed}
The class of matroids with an adjoint is minor-closed.
\end{theorem}

\begin{proof}
    Let $M$ be a matroid with an adjoint, and let $N$ be a minor of $M$. We need to show that $N$ has an adjoint. By \cref{lem:ind-coind}, there are disjoint subsets $C$ and $D$ of $E(M)$ such that $C$ is independent in $M$, such that $D$ is coindependent in $M$, and such that $N=M/C\backslash D$. From \cref{prop:Closed-cont}, it follows that $M/C$ has an adjoint.
    
    Contraction of any set does not change the coindependence of a disjoint set. Therefore, $D$ is a coindependent set in $M/C$. From \cref{lem:closed-del}, it follows that $N=M/C\backslash D$ has an adjoint.
\end{proof}

\section*{Acknowledgments} The author thanks Louis Deaett for introducing him to the notion of an adjoint of a matroid, as well as for many helpful discussions.

The author also thanks Winfried Hochst\"attler for pointing him toward \cite{k95}. Winfried Hochst\"attler also informed the author that, in \cite{k95}, it is shown that, if $\phi$ is an adjoint map from $M$ to $M'$, and if $X$ and $Y$ are flats of $M$, then $\phi(X)$ and $\phi(Y)$ are a modular pair in $M'$. It is claimed that arguments similar to those used in the proof of this fact can be used to show that a minor of a matroid with an adjoint also has an adjoint itself.
\bibliographystyle{alpha}
\bibliography{main.bib}

\end{document}